\documentclass[11pt,twoside,reqno,centertags]{amsart}


\usepackage{amsmath,amsthm,amsfonts,amssymb}
\usepackage[mathscr]{euscript}
\usepackage{hyperref} 
\usepackage{mathrsfs} 
\usepackage{graphicx}
\usepackage{color} 
\usepackage{epsfig}

\DeclareMathOperator*{\esssup}{ess\,sup}
\newtheorem{uess}{Lemma}
\newtheorem{guess}{Theorem}
\newtheorem{corollary}{Corollary}
\newtheorem{proposition}{Proposition}

\newtheorem{example}{Example}
\newtheorem{remark}{Remark}

\begin{document}

\title[Equations with oscillatory kernels]{On exponential stability of linear delay equations with oscillatory coefficients and kernels}
\date{November 18, 2021 - submitted, May 12, 2022 - revised}
\maketitle

\vspace{ -1\baselineskip}

{\small
\begin{center}
 {\sc Leonid Berezansky}   \\
Ben-Gurion University of the Negev,
Beer-Sheva 84105, Israel 
\\[10pt]
{\sc Elena Braverman} \\
University of
Calgary, 2500 University Drive N.W., Calgary, AB,  Canada T2N 1N4
 \end{center}
}

\numberwithin{equation}{section}
\allowdisplaybreaks

 \smallskip

 \begin{quote}
\footnotesize
{\bf Abstract.}  
New explicit exponential stability conditions are presented for the non-autonomous scalar linear functional differential equation 
$$
\dot{x}(t)+ \sum_{k=1}^m a_k(t)x(h_k(t))+\int_{g(t)}^t K(t,s) x(s)ds=0,
$$
where $h_k(t)\leq t$, $g(t)\leq t$, $a_k(\cdot)$ and the kernel $K(\cdot,\cdot)$ are oscillatory and, generally, discontinuous functions.
The proofs are based on establishing boundedness of solutions and later using the exponential dichotomy for linear equations stating that either the homogeneous equation is exponentially stable or a non-homogeneous equation has an unbounded solution for some bounded right-hand side. 
Explicit tests are applied to models of population dynamics, such as controlled Hutchinson and Mackey-Glass equations.
The results are illustrated with numerical examples, and connection to known tests is discussed. 
\end{quote}


\section{Introduction}
\label{introduct}

Equations with oscillatory coefficients and kernels model many
real-world phenomena, such as hematopoiesis with an oscillatory circulation loss rate \cite{Balderrama, Jiang, Zhang3}, 
a neoclassical population growth \cite{Long2,Long} and a computer virus \cite{Tang} models,  
cellular neural networks \cite{Xu}, and an epidemic model  \cite{XuLi}
where an incidence rate is saturated and periodic.
Most of explicit stability tests for a linear functional differential equation (FDE) include
conditions on the sign of coefficients, see, for example, the papers \cite{BB1,BB2,BB5, GH,GP,GD,Kz,SYC,Yoneyama} and the monographs \cite{AS,Hale,KolmMysh}. 
To complement stability results for delay equations with positive coefficients and integral kernels, we consider delay differential equations with, generally, oscillatory coefficients and obtain explicit conditions for uniform exponential stability.

The fundamental function $X(t,s)$ of the scalar ordinary differential equation (ODE)
\begin{equation}
\label{nondelay}
\dot{x}(t)+a(t)x(t) =0
\end{equation}
is a solution of Eq.~\eqref{nondelay}
with the initial condition $X(s,s)=x(s)=1$, it has the form 
 $\displaystyle X(t,s)=e^{-\int_s^t a(\xi)d\xi}$.
Eq.~\eqref{nondelay} is {\em uniformly exponential stable} (UES) if
there are $M>0,\lambda>0$ such that
$$
|X(t,s)| \leq M e^{-\lambda (t-s)} .
$$

An elementary stability result (see e.g. \cite{Li}) is formulated below as a lemma.

\begin{uess}\label{lemma1.1}
Let $a:[0,\infty) \to {\mathbb R}$ be bounded and oscillatory.
Then uniform exponential stability of 
Eq.~\eqref{nondelay}
is guaranteed by anyone of the three hypotheses:
\\
a) there are $a_0>0$ and $t_0\geq 0$ for which 
$\displaystyle \liminf_{t-s\geq a_0} \frac{1}{t-s}\int_s^t a(\tau) d\tau > 0$, $~t\geq s\geq t_0$;
\\
b) there are $a_0>0$, $h>0$ and $t_0\geq 0$ for which $\displaystyle \int_t^{t+h} a(s)ds\geq a_0$, $t\geq t_0$;
\\
c) for some $a_0>0, \alpha_0>0$ and $t_0\geq 0$,
$a(t)=\tilde{a}(t)+\alpha(t)$, where $\displaystyle \tilde{a}(t)\geq a_0>0$ for $t\geq t_0$ and
$\displaystyle \sup_{t\geq s\geq  t_0}\left|\int_s^t \alpha(\xi)d\xi\right|\leq \alpha_0<\infty$.
\end{uess}

It is easy to check that, without essential boundedness of $a$, neither a) nor b) in Lemma~\ref{lemma1.1} imply uniform stability.

If Eq.~\eqref{nondelay} is UES, the function $\int_{t_0}^t e^{-\int_s^t a(\xi)d\xi}ds$  is bounded on $[t_0,\infty)$, which is widely used in justification of estimates in the present paper. Also, exponential estimates of $e^{-\int_s^t a(\xi)d\xi}$ are based on this assumption.

The main object of the paper is a scalar linear equation 
\begin{equation}\label{2.1}
\dot{x}(t)+ \sum_{k=1}^m a_k(t)x(h_k(t))+\int_{g(t)}^t K(t,s) x(s)ds=0.
\end{equation}
For each $t_0\geq 0$ we also introduce a right-hand side
\begin{equation}
\label{2.2}
\dot{x}(t)+ \sum_{k=1}^m a_k(t)x(h_k(t))+\int_{g(t)}^t K(t,s) x(s)ds=f(t), ~~t\geq t_0
\end{equation}
and an initial condition
\begin{equation}
\label{2.3}
x(t)=\varphi(t), ~ t\leq t_0.
\end{equation}
We say that Eq.~(\ref{2.1}) is {\em UES}, if there are
$K>0$ and $\lambda>0$ 
such that any solution $x$ of (\ref{2.1}),(\ref{2.3}) satisfies
\begin{equation}
\label{2.4}
|x(t)|\leq K~e^{-\lambda (t-t_0)}
\left( \sup_{t\leq t_0}|\varphi(t)| \right),~~t\geq t_0,
\end{equation}
where $K, \lambda $ are independent of the choice of the point $t_0$
and of the initial function $\varphi$.

Assume that ODE \eqref{nondelay} with an oscillating coefficient is UES. 
It is generally believed that a small delay, or a delay-coefficient combination, will allow to preserve stability properties of a scalar equation which involves one delay term only
\begin{equation}\label{1.1}
\dot{x}(t)+a(t)x(h(t))=0,~~ t\geq 0.
\end{equation}
It is a natural question whether there is a sufficiently small bound $\tau$ for the delay $0 \leq t-h(t) \leq \tau$ or $a_0$ for the integral 
$\int_{h(t)}^ t |a(s)|ds \leq a_0$ such that, together with exponential stability of \eqref{nondelay}, this estimate
implies exponential stability of equation \eqref{1.1}.
The following example shows that, unlike for Eq.~\eqref{1.1} with $a(t) \geq 0$, such $\tau$ and $a_0$ cannot be found.

\begin{example}
\label{ex_simple}
Consider Eq.~\eqref{1.1} with periodic $a$ and $t-h(t)$ with a period $\varepsilon + \delta$, $a(t) = -1$, $h(t) = t$ for $t\in [(\delta+\varepsilon)n,\varepsilon+(\delta+\varepsilon)n)$ and $a(t) = 1$, $h(t) = n(\delta+\varepsilon)$ for $t\in [\varepsilon+n(\delta+\varepsilon), (n+1)(\varepsilon+\delta))$, $n=0,1,\dots$ and $\delta = e^{\varepsilon}-1$. As $e^x-1>x$ for $x>0$, we have $\delta>\varepsilon$. Thus
\eqref{nondelay} with such $a(t)$ is UES, as condition b) of Lemma~\ref{lemma1.1} holds with $h=\delta+\varepsilon$, $a_0=\delta-\varepsilon>0$. Evidently the solution exponentially grows from $x(0)$ to $x(0)e^{\varepsilon}$ on $[(\delta+\varepsilon)n,\varepsilon+(\delta+\varepsilon)n)$ and then linearly decreases to
$x(0)e^{\varepsilon} - x(0)\delta=x(0)$ at $t=(n+1)(\varepsilon+\delta)$. Thus, \eqref{1.1} is not asymptotically stable, while $\varepsilon$, $\delta$, the maximal delay $\varepsilon + \delta$ and the integral 
$\displaystyle \sup_{t \geq 0} \int_{h(t)}^t |a(s)| \, ds=\varepsilon+\delta$ can be chosen arbitrarily small.
\end{example}

Example~\ref{ex_simple} illustrates that designing sufficient stability tests for delay models with oscillatory coefficients is a very challenging problem even for simplest 
Eq.~\eqref{1.1}. It is not surprising that there are only few papers 
\cite{BB3,BB4,Gil,Gil2,Zhang1,Zhang2} on stability of scalar delay equations with oscillating coefficients. The relation between oscillation and stability for \eqref{1.1} with an oscillating coefficient is discussed in \cite{StavBrav}.
To the best of our knowledge, \cite{Zhang1} is the first paper where explicit stability tests were obtained by application of a fixed-point method.
In \cite{Zhang1}, a nonlinear equation involved both a delay and a non-delay term, its linear counterpart is
\begin{equation}\label{add1}
\dot{x}(t)+a(t)x(t)+b(t)x(h(t))=0.
\end{equation}
The test of \cite[Theorem 2.1]{Zhang1} has the following form for 
Eq.~\ref{add1}.

\begin{proposition}[\cite{Zhang1}]
\label{p6}
Let $a, b, h$ be continuous, $b(t)\geq 0$, $h(t)\leq t$, \\ $\displaystyle \lim_{t\rightarrow\infty} h(t)=\infty$ and
$$
\liminf_{t\rightarrow\infty}\int_0^t a(s)ds>-\infty,~ \int_0^t e^{-\int_s^t a(u)du}b(s)ds\leq \alpha<1, ~~t \geq 0.
$$
Then Eq.~\eqref{add1} is asymptotically stable if and only if $\displaystyle \lim_{t\rightarrow\infty}\int_0^t  a(s)ds=\infty$.
\end{proposition}

For \eqref{1.1} with a constant delay 
\begin{equation}\label{1.5}
\dot{x}(t)+a(t)x(t-\tau)=0,
\end{equation}
measurable $a(t)$ and $\tau >0$, the following statement is a corollary of the main result in \cite{Gil}.

\begin{proposition} [\cite{Gil}] 
\label{p7}
Assume that there exists a constant $b>0$ such that
$$
\sup_{t>0}\left|\int_0^t (a(s)-b)ds\right|<\frac{b}{2b+\sup_{t>0}|a(t)|}
\mbox{~~ and ~~} b<\frac{1}{\tau e} \, .
$$
Then Eq.~\eqref{1.5} is UES.
\end{proposition}

In \cite{BB3} equation \eqref{1.1} with measurable $a$ and $h$ was considered.
The following stability test is a corollary of the main result of the paper.

\begin{proposition}[\cite{BB3}]
\label{p8}
If ODE $\dot{x}(t)+a(t)x(t)=0$ is UES and
$$
\limsup_{r \rightarrow\infty}\left[\left(\sup_{t\geq r}\int_{h(t)}^t |a(s)|ds\right)
\left(\sup_{t\geq r}\int_r^t e^{-\int_s^t a(\tau)d\tau}|a(s)|ds\right)\right]<1,
$$
Eq.~\eqref{1.1} is also UES.
\end{proposition}

The present paper extends the results of \cite{BB3} to scalar equations with several delay terms and an integral term, with coefficients and delays being measurable, and solutions absolutely continuous. We apply these exponential stability results 
to equations of population dynamics.

The structure of the paper is outlined as follows.
Section~\ref{prelim} contains the main results for a generalization of \eqref{1.1} with several terms and a distributed delay.
Based on these statements,  in Section~\ref{applications} we analyze local exponential stability for modified versions of Hutchinson and Mackey-Glass equations with a variable control. Section~\ref{examples} includes examples for linear and nonlinear models, and Section~\ref{conclusion} concludes the paper with a brief summary and overview of possible extensions.

\section{Main Results}
\label{prelim}


For problem~\eqref{2.2},\eqref{2.3} and its special cases, everywhere we assume without further mentioning that
the coefficients $a_k(t)$, the function $f(t)$, as well as the kernel $K(t,s)$, are Lebesgue measurable essentially 
bounded functions in their domains $[0,\infty)$ and $[0,\infty)\times [0,\infty)$, respectively,
$h_k(t)$ and $g(t)$ are measurable delayed arguments with bounded delays
\begin{equation*}
0\leq t-h_k(t)\leq \tau_k\leq \tau=\max_k\tau_k,~~ 0\leq t-g(t)\leq \sigma.
\end{equation*}
In initial condition \eqref{2.3}, $\varphi :[t_0-\max\{\tau,\sigma\},t_0] \rightarrow {\mathbb R}$ is Borel
measurable and bounded.
\vspace{2mm}

By a solution $x:[t_0-\max\{\tau,\sigma\}, \infty) \rightarrow {\mathbb R}$ of \eqref{2.2},\eqref{2.3} 
we mean a function satisfying \eqref{2.2} almost everywhere (a.e.) 
and \eqref{2.3} everywhere, where $x$ is 
locally absolutely continuous on $[0,\infty)$.
All the equalities and inequalities with the derivative are assumed to hold a.e.


As usual, ${\bf L}_{\infty}[t_0,\infty)$ is the space of all measurable essentially
bounded functions $y:[t_0,\infty) \to {\mathbb R}$ with the 
norm 
$\displaystyle 
\|y\|_{[t_0,\infty)}=\esssup_{t\geq t_0} |y(t)|,
$ similarly for any interval $J=[t_0,t_1]$, $t_1>t_0$, 
${\bf L}_{\infty}(J)$ has the norm
$
\|y\|_{J}=\esssup_{t\in J} |y(t)|
$, 
${\bf C}[t_0,\infty)$ is the space of all continuous  bounded functions on  
$[t_0,\infty)$ with the $\sup$-norm. 

The Bohl-Perron type theorem is stated below as a lemma, it is a main tool in stability studies. 

\begin{uess}[\cite{AS}]
\label{lemma2.1}
Let for some $t_0\geq 0$ and any $f\in 
{\bf L}_{\infty}[t_0,\infty)$ the solution of Eq.~\eqref{2.2}
with the initial condition $x(t)=0$, $t\leq t_0$
be in ${\bf C}[t_0,\infty)$. Then, Eq.~(\ref{2.1}) is UES.
\end{uess}

Denote 
\begin{equation}
\label{coef_def}
\begin{array}{l}
\displaystyle
a(t)=\sum_{k=1}^m a_k(t), ~ b(t)=\int\limits_{[g(t),t]\cap [t_0,\infty)} 
\!\!\!\!
K(t,s)ds,~ \overline{b}(t)=\int\limits _{[g(t),t]\cap [t_0,\infty)} \!\!\!\! |K(t,s)|ds, \vspace{2mm} \\
\displaystyle
c(t)=\int\limits_{[g(t),t]\cap [t_0,\infty)} \!\!\!\! (t-s)|K(t,s)|ds, ~~ A=\sum_{k=1}^m \|a_k\|_{[t_0,\infty)}+\|\overline{b}\|_{[t_0,\infty)}.
\end{array}
\end{equation}

\begin{guess}\label{theorem2.1}
Assume that there are  $t_0\geq 0$ and $\alpha \in (0,1)$ such that for $t_0\geq 0$ at least one of the following conditions holds:
\\
a) ODE $\dot{x}(t)+a(t)x(t)=0$ is UES and 
\begin{equation}
\label{2.5}
 \int_{t_0}^t e^{-\int_s^t a(\xi)d\xi}\left[A\sum_{k=1}^m\tau_k|a_k(s)|+\overline{b}(s)\right]ds\leq \alpha<1;
\end{equation}
b) ODE $\dot{x}(t)+b(t)x(t)=0$ is UES and 
\begin{equation*}
\int_{t_0}^t e^{-\int_s^t b(\xi)d\xi}\left[\sum_{k=1}^m|a_k(s)|+Ac(s)\right]ds\leq\alpha<1;
\end{equation*}
c) ODE $\dot{x}(t)+(a(t)+b(t))x(t)=0$ is UES and 
\begin{equation*}
 \int_{t_0}^t e^{-\int_s^t (a(\xi)+b(\xi))d\xi}\left[\sum_{k=1}^m A\left(\sum_{k=1}^m\tau_k|a_k(s)|+c(s)\right)\right]ds\leq\alpha<1.
 \end{equation*}
Then Eq.~(\ref{2.1}) is UES.
\end{guess}

\begin{remark}
Explicit conditions when relevant ODEs are UES can be found in Lemma~\ref{lemma1.1}.
\end{remark}

\begin{proof}
a) Let $t_1>t_0$ be an arbitrary point, $J=[t_0,t_1]$ and $x$ be a solution of Eq.~\eqref{2.2}
assuming the zero initial condition $x(t) \equiv 0$, $t\leq t_0$. From \eqref{2.2} we have
\begin{equation}
\label{2.6b}
\|\dot{x}\|_J\leq A \|x\|_J+ \|f\|_{[t_0,\infty)}.
\end{equation}
Eq.~\eqref{2.2} is transformed to
$$
\dot{x}(t)+a(t)x(t)=\sum_{k=1}^m a_k(t)\int_{h_k(t)}^t \dot{x}(s)ds-\int_{g(t)}^t K(t,s)x(s)ds+f(t),
$$
where $a$ is defined in \eqref{coef_def}.
Hence
\begin{equation*}
x(t)=\int\limits_{t_0}^t e^{-\int_s^t a(\xi)d\xi}\left[\sum_{k=1}^m a_k(s)\int\limits_{h_k(s)}^s \dot{x}(\xi)d\xi-\int\limits_{g(s)}^s K(s,\xi)x(\xi)d\xi\right]ds
+f_1(t).
\end{equation*}
Here $f_1(t)=\int_{t_0}^t e^{-\int_s^t a(\xi)d\xi}f(s)ds \in L_{\infty}[t_0,\infty)$.

Eq.~\eqref{2.2} and inequalities \eqref{2.6b}  and \eqref{2.5} imply 
\begin{align*}
\|x\|_J & \leq \int_{t_0}^t e^{-\int_s^t a(\xi)d\xi}\left[\sum_{k=1}^m \tau_k |a_k(s)| \|\dot{x}\|_J+\overline{b}(s)\|x\|_J\right]ds + \| f_1 \|_{[t_0,\infty)}  
\\
& \leq  \left( \int_{t_0}^t e^{-\int_s^t a(\xi)d\xi}\left[A\sum_{k=1}^m 
\tau_k |a_k(s)|+\overline{b}(s)\right]ds \right) ~\|x\|_J +M_1 
\\ & \leq  \alpha\|x\|_J+M_1,
\end{align*}
where $\displaystyle M_1 = \| f \|_{[t_0,\infty)} \int_{t_0}^{\infty} e^{-\int_s^t a(\xi)d\xi}
\sum_{k=1}^m \tau_k |a_k(s)|~ds + \| f_1 \|_{[t_0,\infty)} < \infty $.

By \eqref{2.5}, $\alpha<1$ and 
 $\|x\|_J\leq \beta<\infty$ for $\beta := M_1/(1-\alpha)$ independent of $t_1$.
Hence $|x(t)|\leq \beta$ on $[t_0,\infty)$. Using Lemma~\ref{lemma2.1}, we conclude that Eq.~\eqref{2.1} is UES.

b) After rewriting Eq.~\eqref{2.2} as 
\begin{align*}
\dot{x}(t)+b(t)x(t) & =  -  \sum_{k=1}^m a_k(t)x(h_k(t))+\int_{g(t)}^t K(t,s)[x(t)-x(s)]ds+f(t)
\\
&= -  \sum_{k=1}^m a_k(t)x(h_k(t))+\int_{g(t)}^t K(t,s)\int_s^t \dot{x}(\xi)d\xi ds+f(t),
\end{align*}
we get
$$
x(t)=\int\limits_{t_0}^t e^{-\int\limits_s^t b(\xi)d\xi}
\left[ -  \sum_{k=1}^m a_k(s)x(h_k(s))+
\!\! \int\limits_{g(s)}^s \!\! K(s,\xi)\int\limits_\xi^s \!\! \dot{x}(\zeta)d\zeta ~d\xi\right]ds
+f_2(t),
$$
where $f_2(t)=\int_{t_0}^t e^{-\int_s^t b(\xi)d\xi}f(s)ds \in L_{\infty}[t_0,\infty)$.

The rest of the proof and justification for Part c) are similar to the scheme for Part a).
\end{proof}

\begin{remark}\label{remark2.1}
The statement of Theorem~\ref{theorem2.1}
can be slightly improved by considering, generally, smaller constants
in \eqref{coef_def}
$$A=\sum_{k=1}^m \|a_k\|_{\Omega_k}+\|\overline{b}\|_{[t_0,\infty)},
~\Omega_k = \{t \geq t_0: h_k(t) >t_0 \}.$$
\end{remark}

If in the proof of  Theorem~\ref{theorem2.1} we substitute $\dot{x}(t)$ from Eq.~\eqref{2.2}  instead of a priori estimation of the derivative in \eqref{2.6b}, we improve the result of Theorem~\ref{theorem2.1}.

\begin{guess}\label{theorem2.1b}
Assume that there are  $t_0\geq 0$ and $\alpha \in (0,1)$ such that for $t_0\geq 0$ at least one of hypotheses  a)-c) holds:
\\
a) ODE $\dot{x}(t)+a(t)x(t)=0$ is UES and 
\begin{equation}
\label{2.5a}
 \int\limits_{t_0}^t e^{-\int_s^t a(\xi)d\xi}\left[\sum_{k=1}^m|a_k(s)|
 \int\limits_{h_k(s)}^s \left(\sum_{k=1}^m|a_k(\xi)|+\overline{b}(\xi)\right)d\xi +\overline{b}(s)\right]ds\leq \alpha<1;
\end{equation}
b) ODE $\dot{x}(t)+b(t)x(t)=0$ is UES and 
\begin{equation*}
 \int\limits_{t_0}^t  \!\! e^{-\int\limits_s^t b(\xi)d\xi} \! \left[\sum_{k=1}^m|a_k(s)|+ \!\!\!
 \int\limits_{g(s)}^s \! \! \! |K(s,\xi)|  \!\!\! \int\limits_{\xi}^s
 \! \left(\sum_{k=1}^m|a_k(\zeta)|+\overline{b}(\zeta)\right)
 \! d\zeta \, d\xi \right]ds\leq\alpha<1;
\end{equation*}
c) ODE $\dot{x}(t)+(a(t)+b(t))x(t)=0$ is UES and 
\begin{equation*}
\begin{array}{ll}
 & \displaystyle 
\int_{t_0}^t e^{-\int_s^t (a(\xi)+b(\xi))d\xi}\left[\sum_{k=1}^m 
|a_k(s)|\int_{h_k(s)}^s \left(\sum_{k=1}^m|a_k(\xi)|
+\overline{b}(\xi)\right)d\xi\right. \\
 + & \displaystyle \left. \int_{g(s)}^s |K(s,\xi)|\int_{\xi}^s 
 \left(\sum_{k=1}^m|a_k(\zeta)|+\overline{b}(\zeta)\right)d\zeta \, d\xi
 \right]ds\leq\alpha<1.
 \end{array}
\end{equation*} 
Then,  Eq.~(\ref{2.1}) is UES.
\end{guess}

Using Lemma~\ref{lemma1.1}, below we give explicit conditions when \eqref{nondelay} is UES and $a$ is a positive coefficient with a small oscillatory perturbation.
The following theorem considers one of the most suitable applications of equations with oscillatory coefficients in real models, where fluctuations arise due to bounded in the integral sense periodic perturbations
of the parameters. In the case a), associated non-delay equation \eqref{nondelay} is UES, as well as relevant non-delay equations in b) and c).

\begin{guess}\label{theorem2.1a}
Let $A,\bar{b}$ and $c$ be denoted in \eqref{coef_def}, and at least one of the hypotheses holds: 
\\
a) $\displaystyle a(t)=\sum_{k=1}^m a_k(t)=\tilde{a}(t)+\alpha(t)$, where $\displaystyle \tilde{a}(t)\geq a_0>0,
~ \sup_{t\geq s\geq  t_0}\left|\int_s^t \alpha(\xi)d\xi\right|\leq \alpha_0<\infty$ and 
\begin{equation}
\label{2.8}
A\sum_{k=1}^m \tau_k\left\|\frac{a_k}{\tilde{a}}\right\|_{[t_0,\infty)}
+\left\|\frac{\overline{b}}{\tilde{a}}\right\|_{[t_0,\infty)} < e^{-\alpha_0}.
\end{equation}
b)
$\displaystyle b(t)=\int\limits_{[g(t),t]\cap [t_0,\infty)} \!\!\!\!\!\!\!\!\!\!\! K(t,s)ds=\tilde{b}(t)+\beta(t)$, where 
$\displaystyle  \tilde{b}(t)\geq b_0>0$, \\ $\displaystyle \sup_{t\geq s\geq  t_0}\left|\int_s^t \beta(\xi)d\xi\right|\leq\beta_0<\infty$ and
\begin{equation*}
\sum_{k=1}^m \left\|\frac{a_k}{\tilde{b}}\right\|_{[t_0,\infty)}
+A\left\|\frac{c}{\tilde{b}}\right\|_{[t_0,\infty)}<e^{-\beta_0}.
\end{equation*}
c) $a(t)+b(t)=d(t)+\gamma(t)$, where $\displaystyle d(t)\geq d_0>0,
~ \sup_{t\geq s\geq  t_0}\left|\int_s^t \gamma(\xi)d\xi\right|\leq \gamma_0<\infty$ and 
\begin{equation*}
A\left(\sum_{k=1}^m \tau_k\left\|\frac{a_k}{d}\right\|_{[t_0,\infty)}
+\left\|\frac{c}{d}\right\|_{[t_0,\infty)}\right) < e^{-\gamma_0}.
\end{equation*}
Then Eq.~\eqref{2.1} is UES.
\end{guess}
\begin{proof}
a)
For the fundamental function of the equation $\dot{x}(t)+a(t)x(t)=0$ we have
$$
|X(t,s)|=e^{-\int_s^t a(\xi)d\xi}=e^{-\int_s^t (\tilde{a}(\xi)+\alpha(\xi))d\xi}
\leq e^{\alpha_0}e^{-a_0(t-s)}.
$$
Thus ODE $\dot{x}(t)+a(t)x(t)=0$ is UES. 
Since
\begin{align*}
 & \int_{t_0}^t e^{-\int_s^t a(\xi)d\xi}\left[A\sum_{k=1}^m\tau_k|a_k(s)|+\overline{b}(s)\right]ds
\\ \leq & \int_{t_0}^t e^{\alpha_0} e^{-\int_s^t \tilde{a}(\xi)d\xi}\tilde{a}(s)\left[A\sum_{k=1}^m\tau_k\frac{|a_k(s)|}{\tilde{a}(s)}
+\frac{\overline{b}(s)}{\tilde{a}(s)}\right]ds
\\
\leq  & e^{\alpha_0}\left(A\sum_{k=1}^m \tau_k\left\|\frac{a_k}{\tilde{a}}\right\|_{[t_0,\infty)}
+\left\|\frac{\overline{b}}{\tilde{a}}\right\|_{[t_0,\infty)}\right)<1,
\end{align*}
Theorem~\ref{theorem2.1} implies that Eq.~\eqref{2.1} is UES.

The proofs of Parts b) and c) are similar.
\end{proof}

Further, we consider special cases of \eqref{2.1}.

\begin{corollary}\label{c2.2}
Let $\displaystyle a(t)=\sum_{k=1}^m a_k(t)=\tilde{a}(t)+\alpha(t)$, where 
$$
\tilde{a}(t)\geq a_0>0, \displaystyle \sup_{t\geq s\geq  t_0}\left|\int_s^t \alpha(\xi)d\xi\right|\leq \alpha_0<\infty
$$
and
\begin{equation*}
\label{2.8b}
\sum_{k=1}^m \|a_k\|_{[t_0,\infty)}\left(\sum_{k=1}^m \tau_k\left\|\frac{a_k}{\tilde{a}}\right\|_{[t_0,\infty)}\right)< e^{-\alpha_0}.
\end{equation*}
Then the equation
$$
\dot{x}(t)+\sum_{k=1}^m a_k(t)x(h_k(t))=0
$$
is UES.
\end{corollary}

\begin{corollary}\label{c2.6}
Let
$\displaystyle b(t)=
\!\! \!\! \!\! \!\! \!\! \!\! 
\int\limits_{[g(t),t]\cap [t_0,\infty)} \!\! \!\! \!\! \!\! \!\! \!\!  K(t,s)ds=\tilde{b}(t)+\beta(t)$,
$\displaystyle \tilde{b}(t)\geq b_0>0$, \\
$\displaystyle \sup_{t\geq s\geq  t_0} \left|\int_s^t \beta(\xi)d\xi\right|\leq \beta_0<\infty$ and
\begin{equation}
\label{c2.3_ineq}
\|\overline{b}\|_{[t_0,\infty)}\left\|\frac{c}{\tilde{b}}\right\|_{[t_0,\infty)}<e^{-\beta_0},
\end{equation}
where $\overline{b}, c$ are denoted in Eq.~\eqref{coef_def}.
Then, the equation
$$
\dot{x}(t)+\int_{g(t)}^t K(t,s) x(s)ds=0
$$
is UES.
\end{corollary}

Consider Eq.~\eqref{2.1} for $m=1$
\begin{equation}
\label{2.13}
\dot{x}(t)+  a(t)x(h(t))+\int_{g(t)}^t K(t,s) x(s)ds=0,
\end{equation}
where $t-h(t)\leq \tau$ and $t-g(t)\leq \sigma$.

\begin{corollary}
\label{c2.7}
Let ODE $\dot{x}(t)+a(t)x(t)=0$ be UES, and for some $t_0\geq 0$ and 
$\alpha \in (0,1)$,
\begin{equation}
\label{2.14}
 \int_{t_0}^t \left. \left. e^{-\int_s^t a(\xi)d\xi}\right[ |a(s)|q(s)+\overline{b}(s) \right]ds\leq \alpha, 
\end{equation}
where $\overline{b}$ is denoted in \eqref{coef_def} and
\begin{equation}
\label{2.14a}
q(t):=\int_{[h(t),t]\cap [t_0,\infty)}\left(|a(s)|+ \overline{b}(s) \right)ds.
\end{equation}
Then, Eq.~\eqref{2.13} is UES.
\end{corollary}

\begin{corollary}\label{c2.5}
Let $a(t)=\tilde{a}(t)+\alpha(t)$, $\tilde{a}(t)\geq a_0>0$, 
$\displaystyle \sup_{t\geq s\geq  t_0}\left|\int_s^t \alpha(\xi)d\xi\right|\leq \alpha_0<\infty$  and
\begin{equation}
\label{2.17}
\tau\left\|\frac{a q}{\tilde{a}}\right\|_{[t_0,\infty)}
+\left\|\frac{\overline{b}}{\tilde{a}}\right\|_{[t_0,\infty)}<e^{-\alpha_0},
\end{equation}
 where $\overline{b}$ is denoted in \eqref{coef_def}, and $q$ in \eqref{2.14a}.
Then, Eq.~\eqref{2.13} is UES.
\end{corollary}

A special case of Eq.~\eqref{2.1} with a non-delay term 
\begin{equation}
\label{2.10}
\dot{x}(t)+ a_0(t)x(t)+\sum_{k=1}^m a_k(t)x(h_k(t))+\int_{g(t)}^t K(t,s) x(s)ds=0
\end{equation}
is considered separately. 

\begin{guess}
\label{theorem2.2}
Let ODE $\dot{x}(t)+a_0(t)x(t)=0$ be UES and for some $t_0\geq 0$ and 
$\alpha \in (0,1)$,  
\begin{equation*}
 \int_{t_0}^t e^{-\int_s^t a_0 (\xi)d\xi}\left[\sum_{k=1}^m|a_k(s)|+\overline{b}(s)\right]ds\leq \alpha,
\end{equation*}
where $\overline{b}$ is denoted in \eqref{coef_def}.
Then Eq.~\eqref{2.10} is UES.
\end{guess}

\begin{corollary}\label{c2.4}
Let $a_0(t)=\tilde{a}(t)+\alpha(t)$, where 
$$\tilde{a}(t)\geq \beta>0,~ \sup_{t\geq s\geq  t_0}\left|\int_s^t \alpha(\xi)d\xi\right|\leq \alpha_0<\infty,$$ 
$\overline{b}$ is denoted in \eqref{coef_def} and
\begin{equation*}
\sum_{k=1}^m \left\|\frac{a_k}{\tilde{a}}\right\|_{[t_0,\infty)}
+\left\|\frac{\overline{b}}{\tilde{a}}\right\|_{[t_0,\infty)}<e^{-\alpha_0}.
\end{equation*}
Then Eq.~\eqref{2.10} is UES.
\end{corollary}

Theorem \ref{theorem2.2} and Corollary~\ref{c2.4} are justified similarly to  Theorems~\ref{theorem2.1}
and \ref{theorem2.1a}. 

\section{Applications}
\label{applications}

Consider a generalized Hutchinson equation with the control-type term
\begin{equation}
\label{3.1}
\dot{N}(t) = N(t) \left. \left. \left.\left. \sum_{j=1}^m r_j(t) \right[ K-N(h_j(t)) \right] - u(t) \right[ N(g(t))-K \right].
\end{equation}
Here the last term corresponds to external impact on the population. If the population  exceeds the carrying capacity $K$   for a period of time, $u(t)>0$ can be interpreted as harvesting.
If the population is less than $K$  for $u(t)>0$, there is external restocking, e.g. adding some juveniles to the pond.
The possibility of sign-changing control $u(t)$, with positive $u$ prevailing, can describe control volatility when
at certain periods of time the control can become destabilizing. In addition to control interpretation, the term including $u$ can describe seasonal movements attracting overpopulated areas for negative $u$ (for example, during mating seasons) and areas with abundant resources for positive $u$.  With seasonality, we can get a $T$-periodic coefficient $u$.

Similar interpretation works for a modified Mackey-Glass model with a control term 
\begin{equation}
\label{3.2}
\dot{N}(t) = r(t)\left[ \frac{a N(h(t))}{1+ N^{\gamma}(h(t))} - b N(t)\right] - u(t) \left[ N(g(t))- N^{\ast} \right], a>b>0,  \gamma > 0,
\end{equation}
where $u(t)$ can be positive or negative, and the unique positive equilibrium is $N^{\ast} = (a/b-1)^{1/\gamma}$.

To apply Corollary~\ref{c2.2} to Eq.~\eqref{3.1} assume that $r_j, u$ are 
measurable essentially bounded for $t \in [t_0,\infty)$, $h_k$ and $g$ are 
measurable, $0\leq t-h_j(t)\leq \tau_j$, $0\leq t-g(t)\leq \sigma$ and  $K>0$.

\begin{guess}\label{theorem3.1}
Let $t_0\geq 0$, $a_0 >0$ and $\alpha_0>0$ be such that for $t\geq t_0$,
$$
r_j(t)=R_j(t)+\bar{r}_j (t), ~~u(t)=U(t)+\bar{u}(t),
$$
\begin{equation}
\label{3.3}
K\sum_{j=1}^m R_j (t) +U(t)\geq a_0, ~~\sup_{t\geq s\geq t_0} \left|\int_s^t \left[K\sum_{j=1}^m \bar{r}_j(\xi) +\bar{u}(\xi)\right]d\xi\right| \leq \alpha_0
\end{equation}
and
\begin{equation}
\label{3.4}
\begin{array}{ll}
& \displaystyle \left( K\sum_{j=1}^m \|r_j\|_{[t_0,\infty)}+\|u\|_{[t_0,\infty)} \right) 
\displaystyle  \left[\sum_{j=1}^m K\tau_j \left\|\frac{r_j}{K\sum_{i=1}^m R_i +U}\right\|_{[t_0,\infty)} \right.
\vspace{2mm} \\
+ & \displaystyle \left. \sigma\left\|\frac{u}{K\sum_{i=1}^m
R_i +U}\right\|_{[t_0,\infty)}\right]
<e^{-\alpha_0}.
\end{array}
\end{equation}
Then, the positive equilibrium $K$ of Eq.~\eqref{3.1} is locally UES.
\end{guess}

\begin{proof}
Substituting $K-N(t)=x(t)$ transforms Eq.~\eqref{3.1} to the model 
\begin{equation}
\label{3.5}
\dot{x}(t) = -(K-x(t))\sum_{j=1}^m r_j(t)  x(h_j(t))  -  u(t)x(g(t))
\end{equation}
with the zero equilibrium. 
After linearization Eq.~\eqref{3.5} becomes
\begin{equation}
\label{3.6}
\dot{y}(t) = -K\sum_{j=1}^m r_j(t)  y(h_j(t))  -u(t)y(g(t)).
\end{equation}
Denote 
$$
a(t)=K\sum_{j=1}^m r_j (t)+ u(t),~~ \tilde{a}(t)=K\sum_{j=1}^m R_j (t)+U(t), ~~~\alpha(t)=K\sum_{j=1}^m \bar{r}_j (t)+\bar{u}(t).
$$
By \eqref{3.3} and \eqref{3.4}, Eq.~\eqref{3.6} satisfies the assumptions of Corollary~\ref{c2.2}.
Hence Eq.~\eqref{3.6} is UES. Then the zero equilibrium of \eqref{3.5} and therefore
the positive equilibrium $K$ of \eqref{3.1}
are locally UES. 
\end{proof}

\begin{corollary}\label{c3.1}
Let $t_0\geq 0$, $a_0 >0$ and $\alpha_0>0$ be such that anyone of the  
hypotheses a)-b) is satisfied for $t\geq t_0$:

a) $r_j(t)=R_j(t)+\bar{r}_j(t),$ 
$$
K\sum_{j=1}^m R_j(t)\geq a_0, ~~\sup_{t\geq s\geq t_0} \left|\int_s^t K\sum_{j=1}^m\bar{r}_j(\xi)d\xi\right| \leq \alpha_0,
$$
\begin{align*}
& \left( K\sum_{j=1}^m \|r_j \|_{[t_0,\infty)}+\|u\|_{[t_0,\infty)} \right) \times
\\ \times & \left[\sum_{j=1}^m \tau_j \left\|\frac{r_j}{\sum_{i=1}^m R_i}\right\|_{[t_0,\infty)}+\sigma\left\|\frac{u}{K\sum_{i=1}^m
R_i}\right\|_{[t_0,\infty)}\right]
<e^{-\alpha_0};
\end{align*}

b) $u(t)=U(t)+\bar{u}(t),$
$$
U(t)\geq a_0, ~~\sup_{t\geq s\geq t_0} \left|\int_s^t\bar{u}(\xi)d\xi\right| \leq \alpha_0,
$$
$$
\left( K\sum_{j=1}^m \|r_j\|_{[t_0,\infty)}+\|u\|_{[t_0,\infty)} \right)
\left[\sum_{j=1}^m K\tau_j\left\|\frac{r_j}{U}\right\|_{[t_0,\infty)}+\sigma\left\|\frac{u}{U}\right\|_{[t_0,\infty)}\right]
<e^{-\alpha_0}.
$$
Then the positive equilibrium $K$ of Eq.~\eqref{3.1} is locally UES.
\end{corollary}

Consider now Eq.~\eqref{3.2}.

Let $r(t)\geq 0$ be 
essentially bounded on $[t_0,\infty)$,
$0\leq t-h(t)\leq \tau$ and $ 0\leq t-g(t)\leq \sigma$. 
Denote
$\displaystyle \nu :=\frac{b}{a}[\gamma (a-b)-a]$, recall that $N^*=(a/b-1)^{1/\gamma}$  and note that $\nu+b=b\gamma(a-b)/a>0$.

\begin{guess}\label{theorem3.2}
Let $t_0\geq 0$, $r_0>0$ and $\alpha_0>0$ be such that for $t\geq t_0$,
\begin{equation}
\label{3.7}
r(t)\geq r_0>0, ~~\sup_{t\geq s\geq t_0} \left|\int_s^t u(\xi)d\xi\right| \leq \alpha_0,
\end{equation}
\begin{equation}
\label{3.8}
\left( (b+\nu)\|r\|_{[t_0,\infty)}+\|u\|_{[t_0,\infty)}\left)\left(\frac{\tau\nu}{b+\nu}+\frac{\sigma}{b+\nu}\left\|\frac{u}{r}\right\|_{[t_0,\infty)}\right) \right.\right.
<e^{-\alpha_0}.
\end{equation}
Then the positive equilibrium $N^*$ of Eq.~\eqref{3.2} is locally UES.
\end{guess}
\begin{proof}
After the substitution $N(t)=N^*+x(t)$, Eq.~\eqref{3.2} transforms to
\begin{equation}
\label{3.9}
\dot{x}(t)=r(t)\left[\frac{a(N^*+x(h(t)))}{1+[N^*+x(h(t))]^{\gamma}}-b(N^*+x(t))\right]-u(t)x(g(t))
\end{equation}
with the zero equilibrium.
Since $b=\frac{a}{1+(N^*)^{\gamma}}$, Eq.~\eqref{3.9} is equivalent to
\begin{equation}
\label{3.10}
\dot{x}(t)=r(t)\left[a\left(\frac{N^*+x(h(t))}{1+[N^*+x(h(t))]^{\gamma}}-\frac{N^*}{1+(N^*)^{\gamma}} \right)-bx(t)\right]-u(t)x(g(t)).
\end{equation}
The function
$$
g(x)=a\left(\frac{N^*+x}{1+(N^*+x)^{\gamma}}-\frac{N^*}{1+(N^*)^{\gamma}}\right)
$$
satisfies
$
g(0)=0$, $\displaystyle g'(x)=a\frac{1+(N^*+x)^{\gamma}-\gamma(N^*+x)^{\gamma}}{(1+(N^*+x)^{\gamma})^2}$, \\ $\displaystyle 
g'(0)=a\frac{1+(1-\gamma)(N^*)^{\gamma}}{(1+(N^*)^{\gamma})^2}=-\nu.
$

Hence the linearization of \eqref{3.10} is
\begin{equation}
\label{3.11}
\dot{y}(t)=-\nu r(t)y(h(t))-br(t)y(t)-u(t)y(g(t)).
\end{equation}
Let
$\displaystyle
a(t)=(\nu+b)r(t)+u(t),~ \tilde{a}(t)=(\nu+b)r(t), ~\alpha(t)=u(t).
$
As $\nu+b>0$, we get $(\nu+b)r(t)>(\nu+b)r_0>0$. 
By \eqref{3.7},\eqref{3.8} and Corollary \ref{c2.2}, Eq.~\eqref{3.11} is UES.
Hence the zero equilibrium of \eqref{3.10} is locally UES.
Therefore the solution $N^*$ of Eq.~\eqref{3.2} is locally UES.
\end{proof}

\begin{remark}
\label{remark_Mackey}
It is easy to see that Theorem~\ref{theorem3.2} follows the lines of Corollary~\ref{c3.1}, Part a) with $\tilde{r}_k \equiv 0$. It is possible to get more general results, as in Theorem~\ref{theorem3.1}.
\end{remark}

\section{Examples}
\label{examples}

\begin{example}
\label{example4.1}
The delay equation 
\begin{equation}
\label{4.1}
\begin{array}{ll}
\dot{x}(t) & +0.1(0.9 - \sin 5t)x(t-0.1)+0.1(0.9+ \cos 10t)x(t-0.2) \\ & \displaystyle+\int_{g(t)}^t 
\left[ 0.1+0.2\cos(t-s) \right]\, x(s)ds=0,~~t-g(t)\leq \sigma,
~~t\geq 0
\end{array}
\end{equation}
has oscillatory coefficients and kernel.
Applying  Theorem~\ref{theorem2.1a} a), let us find the bound for the parameter $\sigma$ for which Eq.~\eqref{4.1} is UES. Let us estimate the parameters 
$$
a(t)=a_1(t)+a_2(t)=0.1(0.9 - \sin 5t)+0.1(0.9+ \cos 10t),~~ \tilde{a}(t)=0.18, 
$$$$
\alpha(t)=-0.1\sin 5t+0.1 \cos 10t,~~\overline{b}(t)=\int_{[g(t),t]\cap [0,\infty)} |0.1+0.2\cos(t-s)|ds,
$$$$
 \alpha_0=\sup_{t> s}\left|\int_s^t\alpha (\xi)d\xi\right|\leq 2\left(\frac{0.1}{5}+\frac{0.1}{10}\right)=0.06,
$$$$
\|a_1\|_{[0,\infty)}=\|a_2\|_{[0,\infty)}=0.19,~\|\overline{b}\|_{[0,\infty)}\leq 0.3\sigma,~A\leq 0.38+0.3\sigma.
$$
Inequality \eqref{2.8} holds if 
$$
(0.38+0.3\sigma) 0.3 \frac{0.19}{0.18} +\frac{0.3\sigma}{0.18} <e^{-0.06},
$$
which is true for $0<\sigma\leq 0.46628$. 
Hence Eq.~\eqref{4.1} is UES for  $0<\sigma\leq 0.46628$.
\end{example}

\begin{example}\label{example4.2} 
For the integro-differential equation
\begin{equation}
\label{4.2}	
\dot{x}(t)+\int_{t-0.1}^t [10+11\sin(t-10s)]x(s)ds=0, \quad t\geq 0,
\end{equation}
to check stability 
let us apply Corollary~\ref{c2.6}. We have
\begin{align*}
b(t)= & \int_{t-0.1}^t [10+11\sin(t-10s)]ds=1+1.1 (\cos 9t-\cos(9t-1))
\\
= &1-2.2 \sin (0.5) \sin(9t-0.5),
\\
\overline{b}(t)= &\int_{t-0.1}^t |10+11\sin(t-10s)|ds, ~~
\|\overline{b}(t)\|_{[t_0,\infty)}\leq 2.1.
\end{align*}
Introducing $\tilde{b}(t)=1$ and $\beta(t)=-2.2 \sin (0.5) \sin(9t-0.5)$, 
we get 
$$
\int_s^t \beta(\xi)d\xi=  -\int_s^t 2.2 \sin(0.5) \sin(9\xi-0.5) d\xi, $$ $$ 
\left|\int_s^t \beta(\xi)d\xi\right|\leq 4.4 \sin(0.5)/9\approx 0.2343858 =:\beta_0,
$$
$$
c(t)= \int_{t-0.1}^t (t-s)|10+11\sin(t-10s)|ds, $$ $$
|c(t)|\leq 21\int_{t-0.1}^t (t-s)ds=21 \frac{0.1^2}{2}=0.105.
$$
Hence
$$
e^{\beta_0}\|\overline{b}\|_{[t_0,\infty)}\left\|\frac{c}{\tilde{b}}\right\|_{[t_0,\infty)}\leq e^{0.2344}2.1\frac{0.105}{1}\approx 0.2787<1.
$$
Thus inequality  \eqref{c2.3_ineq} holds, and Eq.~\eqref{4.2}  is UES.
\end{example}

\begin{example}
\label{example4.3}
For the model with a discontinuous sign-changing coefficient
\begin{equation}
\label{4.3}	
\dot{x}(t)+a(t) x(t-\tau)=0, \quad t\geq 0,
\end{equation}
where
$\displaystyle
 a(t)=\left\{\begin{array}{rl}
\mu,& t\in [2n,2n+1),\\
-\beta,& t\in [2n+1,2n+2),\\
\end{array}\right.
$ ~
$n=0,1,2, \dots
$ ~
and $0<\beta<\mu$, 
let us apply Corollary~\ref{c2.2} for $m=1$.

Denote for $\beta<\lambda<\mu$ the positive and oscillating parts
$$
\tilde{a}(t)=
\left\{\begin{array}{rl}
\mu-\lambda, &t\in [2n,2n+1), \\
\lambda-\beta,&t\in [2n+1,2n+2),\\
\end{array}
\right.
\!\!
\alpha(t) =
\left\{\begin{array}{rl}
\lambda,&t\in [2n,2n+1), \\
-\lambda,&t\in [2n+1,2n+2), \\
\end{array}\right.
$$
$n=0,1,2, \dots~$,
respectively.

We have $\tilde{a}(t)\geq \min\{\mu-\lambda,\lambda-\beta\}>0$ and 
$a(t)=\tilde{a}(t)+\alpha(t)$. Applying Corollary \ref{c2.2}, we get
$$
\left|\int_s^t \alpha(\xi)d\xi\right|\leq \lambda:=\alpha_0,~ \|a\|_{[t_0,\infty)}=\mu,~ \left\|\frac{a}{\tilde{a}}\right\|_{[t_0,\infty)}
=\max  \left\{ \frac{\mu}{\mu-\lambda},\frac{\beta}{\lambda-\beta}  \right\}.
$$

Hence, if there exists $\lambda \in (\beta,\mu)$ such that
\begin{equation} \label{example_cond1}
\tau\mu \max\left\{\frac{\mu}{\mu-\lambda},\frac{\beta}{\lambda-\beta}\right\}e^{\lambda}<1,
\end{equation}
Eq.~\eqref{4.3} is UES.

For a sufficient explicit stability test, choose $\lambda$ as an arithmetic mean $\lambda=(\mu+\beta)/2$.
We have
$$
\max\left\{\frac{\mu}{\mu-\lambda},\frac{\beta}{\lambda-\beta}\right\}=\frac{2\mu}{\mu-\beta}.
$$
Hence the inequality
\begin{equation} 
\label{add_1}
\frac{2\tau\mu^2}{\mu-\beta}e^{(\mu+\beta)/2} < 1
\end{equation}
implies that Eq.~\eqref{4.3} is UES. 

 Similarly, if we choose the harmonic mean $\lambda=\frac{2\mu \beta}{\mu+\beta}$ then 
$$
 \max\left\{\frac{\mu}{\mu-\lambda},\frac{\beta}{\lambda-\beta}\right\}=\frac{\mu+\beta}{\mu-\beta}.
$$ 
Therefore the condition 
\begin{equation}
\label{add_2}
\tau \mu \frac{\mu+\beta}{\mu-\beta} e^{\frac{2\mu \beta}{\mu+\beta}} < 1
\end{equation}
also implies exponential stability of Eq.~\eqref{4.3}.
\end{example}

\begin{example}
\label{example4.4}
The modified Hutchinson equation for $t\geq 0$
\begin{equation}
\label{4.4}	
\begin{array}{ll}
\dot{N}(t)= & \displaystyle N(t)(0.5-0.75\cos(10t))(1-N(t-h_0|\sin (t)|) )
\vspace{2mm} \\ & \displaystyle -(0.5+0.75\sin(10t))(N(t-0.1\cos^2(t))-1)
\end{array}
\end{equation}
has locally UES positive equilibrium $K=1$ if the assumptions of Theorem~\ref{theorem3.1} hold for $m=1$. Let $K=1$, 
$r(t)=0.5-0.75\cos(10t)$, $R(t)=0.5$, $\bar{r}(t)=-0.75\cos(10t)$, $u(t)=0.5+0.75\sin(10t)$,
$U(t)=0.5,\bar{u}(t)=0.75\sin(10t)$, $h(t)=t-h_0|\sin (t)|,h_0>0$, $g(t)=t-0.1\cos^2(t)$.

Since
$$
R(t)+U(t)=1, ~~ \|r\|_{[0,\infty)}=1.25, ~~\|u\|_{[0,\infty)} =1.25 
$$ 
and
$$
\tau=h_0,\sigma=0.1,
~~\sup_{t\geq s\geq 0}\left|\int_s^t (\bar{r}(\xi)+\bar{u}(\xi))d\xi\right|\leq 0.15\sqrt{2}=\alpha,
$$
the conditions of Theorem~\ref{theorem3.1} hold if
$3(h_0\cdot 1.5+0.1\cdot 1.5)<e^{-0.15\sqrt{2}} $ which is satisfied for $h_0<\left(\frac{e^{-0.15\sqrt{2}}}{3}-0.15\right)/1.5$. Hence the equilibrium $K=1$ of
Eq.~\eqref{4.4} is locally UES if $h_0<0.0797$.
\end{example}

\section{Discussion}
\label{conclusion}

Let us recall that a scalar ODE  $\dot{x}(t)+a(t)x(t)=0$ is UES if the coefficient $a(t)$ is not necessarily positive,
but has to be positive in some integral sense. For linear FDEs, such as  Eq.~\eqref{1.1}, 
a usual assumption in stability tests is $a(t)\geq 0$, which does not hold for many real-world models with seasonal fluctuations.
In the present paper we obtained explicit tests when linear FDEs without assumptions on signs of coefficients are UES and deduced local UES conditions for controlled Hutchinson and  Mackey-Glass equations.
The results of the paper are illustrated with several examples.

Now proceed to comparison with known stability results for equations with oscillating coefficients. 
 
Proposition~\ref{p7} and Corollary~\ref{c2.2} can be compared using Example \ref{example4.3}. 
The proposition is applicable only for $b=\frac{\mu+\beta}{2}$. Other stability conditions in the proposition after some routine calculations give the following stability test for Eq.~\eqref{1.1}
\begin{equation}\label{5.4}
 \frac{(\mu-\beta)(2 \mu +\beta)}{\mu + \beta}<1,~~ e\tau\frac{\mu+\beta}{2}<1.
\end{equation}

To compare, let us fix $\mu = 11 \beta$. Then \eqref{5.4} becomes
$$\beta < \frac{6}{115}, ~~ \tau \beta < \frac{1}{6e} \, ,$$
while conditions \eqref{add_1} and \eqref{add_2} are
$$ 
\tau \beta e^{6 \beta} < \frac{5}{121} \quad \mbox{ and } \quad
\tau \beta e^{11 \beta/6} < \frac{1}{66} \, ,
$$
respectively.
Thus, for $\beta \geq \frac{6}{115} \approx 0.05217$, the criterion in Proposition~\ref{p7} cannot be applied. If we take $\beta = 1$, the bound for the delay $\tau < \tau^* \approx 0.0024$ given by \eqref{add_2} is sharper than $\tau < \tau^* \approx 0.0001$ described in \eqref{add_1}. For $\beta = 0.1$, the bound for the delay $\tau < \tau^* \approx 0.22678$ defined in \eqref{add_1} is sharper than $\tau < \tau^* \approx 0.1261$ in \eqref{add_2}. Further, for $\beta = 0.05$, already \eqref{5.4} holds for $\tau < 10/3$.
The bounds of $\tau^* \approx 0.61$ and $\tau^* \approx 0.276$ are worse in the case when Proposition~\ref{p7} works.

 Comparing Proposition~\ref{p8} with Theorem~\ref{theorem2.1b}, we observe
 that inequality \eqref{2.5a} implies  (for $m=1$ and without the integral term)
 the statement of Proposition~\ref{p8}, thus Theorem~\ref{theorem2.1b} generalizes this proposition.
Application of Theorem~\ref{theorem2.1}  to Eq.~\eqref{1.1} gives the stability test
%
  $$
  \tau \|a\|_{[t_0,\infty)}\int_{t_0}^t e^{-\int_s^t a(\xi)d\xi}|a(s)|ds\leq \alpha<1
  $$
which is less sharp than Proposition~\ref{p8}.
However, Theorems~\ref{theorem2.1} and \ref{theorem2.1b} are applicable 
 to more general non-autonomous equations than \eqref{1.1}. 
 
In addition, let us dwell on relevant areas and ideas
for future relevant research on equation and systems, whether with or without integral terms, including retarded arguments and oscillatory coefficients and/or kernels.
\begin{enumerate}
\item
It would be natural to extend results in the obvious direction when the reference exponentially stable ODE includes a partial sum of $a_k$, or some other combination of concentrated-distributed terms, as a coefficient. Also, a FDE can be used as a reference equation rather than ODE, see, for example, \cite{BB2007} where this idea was implemented.
\item
It is a challenging problem to investigate stability for systems of scalar linear delay differential equations with oscillating coefficients, as well as vector delay differential equations without any sign conditions on the entries of the matrices.
\item
For applied nonlinear equations with oscillating coefficients considered in the present paper we obtained local stability conditions. It would be interesting to explore global stability.
\item
For explicit UES conditions for all equations considered in the paper, we use only Part c) of Lemma \ref{lemma1.1}. It is also possible to obtain exponential stability tests applying other parts of the lemma.
\end{enumerate}

\section*{Acknowledgment}

E. Braverman was partially supported by Natural Sciences and Engineering Research Council of Canada, the grant number is RGPIN-2020-03934 in the framework of the Discovery Grant program. The authors are grateful to the anonymous referee whose thoughtful comments contributed to the current form of the paper.

\end{document}